\documentclass[12pt]{amsart}
\usepackage{amsmath,amsthm,amsfonts,latexsym,amssymb,mathrsfs}

\newtheorem{thm}{Theorem}

\newtheorem{lem}[thm]{Lemma}

\newtheorem{cor}[thm]{Corollary}
\numberwithin{thm}{section}
\numberwithin{equation}{section}

\theoremstyle{definition}

\newcommand{\rat}{\mathbb Q}

\newcommand{\alg}{\overline\rat}

\newcommand{\proj}{\mathbb P}
\newcommand{\intg}{\mathbb Z}
\newcommand{\boldx}{{\bf x}}

\newcommand{\boldz}{{\bf z}}

\newcommand{\bolda}{{\bf a}}

\newcommand{\boldc}{{\bf c}}
\newcommand{\bolde}{{\bf e}}
\newcommand{\boldw}{{\bf w}}

\title[Estimating heights using auxiliary functions]
{Estimating heights using auxiliary functions}
\author[C.L. Samuels]{Charles L. Samuels}
\address{Max-Planck-Instit\"ut f\"ur Mathematik, Vivatsgasse 7, 53111 Bonn, Germany}
\email{csamuels@mpim-bonn.mpg.de}
\subjclass[2000]{Primary 11R04, 11R09}
\keywords{projective height, Weil height, subspace height, Mahler measure, Lehmer's problem}

\begin{document}

\baselineskip=17pt

\begin{abstract}
  Several recent papers construct auxiliary polynomials to bound the Weil height of certain classes of algebraic 
  numbers from below.  Following these techniques, the author gave a general method for introducing auxiliary 
  polynomials to problems involving the Weil height.  The height appears as a solution to a certain extremal 
  problem involving polynomials.  We further generalize the above techniques to acquire both the projective 
  height and the height on subspaces in the same way.  We further obtain lower bounds on the heights of points
  on some subvarieties of $\proj^{N-1}(\alg)$.
\end{abstract}

\maketitle

\section{Introduction}

Let $K$ be a number field and let $v$ be a place of $K$ dividing the place $p$ of $\rat$.  Of course, if $v$ is
non-Archimedean then $p$ is a rational prime while if $v$ is Archimedean then $p=\infty$.  We write $K_v$
to denote the completion of $K$ at $v$ and $\rat_p$ to denote the completion of $\rat$ at $p$.  It is clear
that these completions do not depend on a specific absolute value taken from the places $v$ and $p$.  We write
$d=[K:\rat]$ for the global degree and $d_v=[K_v:\rat_p]$ for the local degree.

We now select two absolute values on $K_v$ for each place $v$.  The first absolute value, denoted $\|\cdot \|_v$, is
the unique extension of the $p$-adic absolute value on $\rat_p$.  The second, denoted $|\cdot |_v$, is defined by
\begin{equation*}
  |x|_v = \|x\|_v^{d_v/d}
\end{equation*}
for all $x\in K_v$.  We note the important identity
\begin{equation*}
  d = \sum_{v|p}d_v
\end{equation*}
as well as the product formula
\begin{equation*}
  \prod_v|\alpha|_v = 1
\end{equation*}
for all $\alpha\in K^\times$.  Furthermore, each of the above absolute values extends uniquely to an 
algebraic closure $\overline K_v$.  If $v$ is Archimedean then $\overline K_v$ is complete, however,
in general, $\overline K_v$ is not complete and we write $\Omega_v$ to denote
its completion.  It is well-known that $\Omega_v$ is algebraically closed for all places $v$.  
Moreover, we may define the {\it Weil Height} of $\alpha\in K$ by
\begin{equation*}
  h(\alpha) = \prod_v\max\{1,|\alpha|_v\}
\end{equation*}
where the product is taken over all places $v$ of $K$.  By the way we have normalized our absolute values, 
this definition does not depend on $K$, and therefore, is a well-defined function on $\alg$.

For $f\in\intg[x]$ having roots $\alpha_1,\ldots,\alpha_d$ we define the {\it Mahler measure} of $f$ by
\begin{equation*}
  \mu(f) = \prod_{k=1}^dh(\alpha_k).
\end{equation*}
Since $h$ is invariant under Galois conjugation over $\rat$, we note that if $f$ is irreducible and
$\alpha$ is any root of $f$ then $\mu(f) = h(\alpha)^{\deg\alpha}$.

By Kronecker's Theorem, $\mu(f)\geq 1$ with equality precisely when $f$ is a product of cyclotomic polynomials
and $\pm x$.  Further, in 1933, D.H. Lehmer \cite{Lehmer} asked if there exists a constant $c>1$ such that 
$\mu(f)\geq c$ in all other cases.  It can be computed that
\begin{equation*}
  \mu(x^{10} + x^9 - x^7 - x^6 - x^5 - x^4 - x^3 + x +1) = 1.17\ldots
\end{equation*}
which remains the smallest known Mahler measure greater than $1$.

Since Lehmer's famous 1933 paper, many special cases of his proposed problem have been solved.  
In 1971, Smyth \cite{Smyth} showed that if $\alpha$ and $\alpha^{-1}$ are not Galois conjugates, then the minimal
polynomial of $\alpha$ over $\rat$ has Mahler measure at least $\mu(x^3-x-1)$.  In a different direction,
Schinzel \cite{Schinzel} showed as corollary to a more general result that if $f\in\rat[x]$ has only real roots then 
$\mu(f)\geq (1+\sqrt 5)/2$. 

Recently, Borwein, Dobrowolski and Mossinghoff \cite{BDM} showed that if $f\in\intg[x]$ has no cyclotomic factors
and has coefficients congruent to $1$ modulo an integer $m$, then
\begin{equation} \label{BDMAuxPol}
  \mu(f)\geq c_m(T)^{\frac{\deg f}{1+\deg f}}.
\end{equation}
Here, $c_m(T)>1$ is a constant depending on $m$ and an auxiliary polynomial $T\in\intg[x]$.  They were
able to obtain an explicit lower bound for $\mu(f)$ by making a choice of auxiliary polynomial $T$.
Later, Dubickas and Mossinghoff \cite{DubMoss} generalized the results of \cite{BDM} so that the polynomial
$f$ in \eqref{BDMAuxPol} may be any factor of a polynomial having coefficient congruent to $1\mod m$.
They further constructed a sequence of auxiliary polynomials that further improved the explicit bounds
given in \cite{BDM}.    Following these methods, the author \cite{Samuels} constructed a function
$U(\alpha,T)$ and showed that
\begin{equation} \label{SamuelsOrig}
  1= h(\alpha)\cdot U(\alpha,T),
\end{equation}
for all polynomials $T$ over $\alg$ with $T(\alpha)\ne 0$.  We now briefly recall this construction.

Define the {\it local supremum norm} of $T\in\Omega_v[x]$ on the unit ball by
\begin{equation} \label{PolyLocalSupNorm}
  \nu_v(T) = \sup\{|T(z)|_v: z\in\Omega_v\ \mathrm{and}\ |z|_v\leq 1\}.
\end{equation}
Consider the vector space $\mathcal J_v$ of polynomials over $\Omega_v$ of degree at most $N-1$.  For
$\alpha\in\Omega_v$ and $T\in \mathcal J_v$ define
\begin{equation} \label{PolyQuotientSupNorm}
  U_v(\alpha,T) = \inf\{\nu_v(T-f): f\in\mathcal J_v\ \mathrm{and}\ f(\alpha) = 0\}.
\end{equation}
Lemma 2.1 of \cite{Samuels} states that
\begin{equation} \label{SamuelsOrigLem}
  |T(\alpha)|_v = \max\{1,|\alpha|_v\}^N\cdot U_v(\alpha,T).
\end{equation}
Now assume that $\alpha\in\alg$ and $T\in\alg[x]$ with $T(\alpha)\ne 0$.  In this situation,
we may define a global version of $U_v(\alpha,T)$ by
\begin{equation*}
  U(\alpha,T) = \prod_vU_v(\alpha,T)
\end{equation*}
where the product runs over all places $v$ of a number field containing $\alpha$ and the coefficients of $T$.
According to \eqref{SamuelsOrigLem}, this product is indeed finite and it does not depend on the number 
field we choose to contain $\alpha$ and the coefficients of $T$.  We may apply the product formula to
$|T(\alpha)|_v$ to obtain \eqref{SamuelsOrig}.  

The advantage of this identity is that we may freely select $T$
in a way that is convenient without changing the value of $U(\alpha,T)$.  It can then be used to estimate the
Weil height in certain special cases as found in \cite{Samuels}.
Our goal for this paper is to apply this strategy to obtain analogous results regarding the projective 
height and, more generally, the height on subspaces.

If $\bolda = (a_1,\ldots,a_N)\in \Omega_v^N$ define the {\it local projective height} of $\bolda$ by
\begin{equation} \label{LocalProjHeight}
  H_v(\bolda) = \max\{|a_1|_v,\ldots,|a_N|_v\}.
\end{equation}
That is, the local projective height is simply the maximum norm on $\Omega_v^N$ with respect to $|\cdot |_v$.
It is worth noting that some authors define the projective height using the maximum norm only at
the non-Archmedean places while using the $L^2$ norm on the components of $\bolda$ at the Archimedean places.
However, we are motivated by generalizing the Weil height, so we will find it more relevant to use the maximum norm 
at all places in our definition. Indeed, we note that 
\begin{equation*}
  H((1,\alpha,\ldots,\alpha^N)) = h(\alpha)^N.
\end{equation*}
It is clear that $H_v(\bolda)=1$ for almost all places $v$ of $K$ so we may define 
the {\it global projective height} of $\bolda\in K^N$ by
\begin{equation*}
  H(\bolda) = \prod_vH_v(\bolda)
\end{equation*}
where the product is taken over all places $v$ of $K$.  Of course, by the way we have chosen our absolute values,
this definition does not depend on $K$.  Furthermore, the product formula implies that $H(\bolda)$ is well
defined on $\proj^{N-1}(\alg)$.  In section \ref{Homogeneous}, we define $U(\bolda,T)$ analogous to
\eqref{PolyQuotientSupNorm} and prove that
\begin{equation} \label{ProjTheorem}
  1 = H(\bolda)^M\cdot U(\bolda,T).
\end{equation}  
Here $T$ is an homogeneous polynomial of degree $M$ in $N$ variables over $\alg$ with $T(\bolda)\ne 0$.
We also give a simple application of this result to demonstrate how it might be applied.

In a slightly different direction, suppose that $W$ is a subspace of $\alg^N$ with basis 
$\{\boldw_1,\ldots,\boldw_M\}$.  The {\it height of $W$} is defined to be the height of the vector 
$\boldw_1\wedge\cdots\wedge \boldw_M$ in the exterior product $\wedge^M(\alg^N)$.  That is,
\begin{equation} \label{GlobalMatrixHeight}
  H(W) = H(\boldw_1\wedge\cdots\wedge \boldw_M).
\end{equation}
This definition does not depend on $K$, and it follows 
from the product formula that $H(W)$ does not depend on our choice of basis.
In section \ref{Matrices}, we define $U(W,\Psi)$ for a surjective linear transformation $\Psi:\alg^N\to\alg^M$
and prove that
\begin{equation} \label{SubspaceTheorem}
  1= H(W)\cdot U(W,\Psi)
\end{equation}
whenever $W\cap\ker\Psi = \{{\bf 0}\}$.  This provides an analog of \eqref{SamuelsOrig} using the height on
subspaces.

\section{The projective height using auxiliary homogeneous polynomials} \label{Homogeneous}

We begin by defining the function $U(\bolda,T)$ given in \eqref{ProjTheorem}.
Let $\mathcal L_v$ denote the vector space of homogeneous polynomials over $\Omega_v$ of degree $M$ in $N$ 
variables along with the zero polynomial.   We define an analog of the local supremum norm on polynomials by
\begin{equation} \label{ProjectiveLocalSupHeight}
  \nu_v(T) = \sup\{|T(\boldz)|_v:\boldz\in\Omega_v^M, H_v(\boldz)\leq 1\}
\end{equation}
and set
\begin{equation} \label{TwistedProjectiveLocalSupHeight}
  U_v(\bolda,T) = \inf\{\nu_v(T-f):f\in \mathcal L_v,\ f(\bolda)=0\}
\end{equation}
for $T\in\mathcal L_v$.  This is the local version of $U(\bolda,T)$ that will appear in our theorem.
Let
\begin{equation*}
	Z(\bolda) = \{f\in \mathcal L_v: f(\bolda) = 0\}.
\end{equation*}
It is obvious that \eqref{TwistedProjectiveLocalSupHeight} descends to a norm on the one-dimensional quotient
$\mathcal L_v/Z(\bolda)$ so that the ratio $|T(\bolda)|_v/U_v(\bolda,T)$ does not depend on $T$.  
In fact, we are able to prove something much stronger.

\begin{lem} \label{ProjectiveLocal}
  If $\bolda\in\Omega_v^N$ then
  \begin{equation} \label{ProjectiveLocalEq}
    |T(\bolda)|_v = H_v(\bolda)^M\cdot U_v(\bolda,T)
  \end{equation}
  for all $T\in\mathcal L_v$.
\end{lem}
\begin{proof}
 We will assume that $|a_n|_v = H_v(\bolda)$ and note that
  \begin{equation*}
    |T(\bolda)|_v = |a_n|_v^M\cdot \left|T\left(\frac{\bolda}{a_n}\right)\right|_v
    \leq H_v(\bolda)^M\cdot\nu_v(T)
  \end{equation*}
  for all homogeneous polynomials $T$ of degree $M$ in $N$ variables.  So if $f(\bolda) = 0$ then
  \begin{equation}\label{UBound1}
    |T(\bolda)|_v = |T(\bolda) - f(\bolda)|_v \leq H_v(\bolda)\cdot\nu_v(T-f).
  \end{equation}
  Taking the infimum of the right hand side \eqref{UBound1} over all $f$ having $f(\bolda)=0$ we obtain
  \begin{equation} \label{LEQ2}
    |T(\bolda)|_v \leq H_v(\bolda)\cdot U_v(\bolda,T).
  \end{equation}
  We now attempt to establish the opposite inequality.  We have that
  \begin{align*}
    U_v(\bolda,T) & = \inf\{\nu_v(T-f):f\in Z(\bolda)\} \\
    & = \inf\{\nu_v(T(\boldz)-(T(\boldz)-(T(\bolda)(z_n/a_n)^M)-T(\bolda)f(\boldz)):f\in Z(\bolda)\} \\
    & = \inf\{\nu_v(T(\bolda)(z_n/a_n)^M-T(\bolda)f(\boldz)):f\in Z(\bolda)\} \\
    & = |T(\bolda)|_v\cdot U_v(\bolda,(z_n/a_n)^M).
  \end{align*}
  It is clear that 
  \begin{align*}
    U_v(\bolda,(z_n/a_n)^M) & \leq \nu_v((z_n/a_n)^M) \\
    & = \sup\{|z_n/a_n|_v^M:|z_n|\leq 1\} \\
    & = |a_n|_v^{-M}
  \end{align*}
  and hence
  \begin{equation*}
    U_v(\bolda,T) \leq |T(\bolda)|_v\cdot H_v(\bolda)^{-M}
  \end{equation*}
  which completes the proof.
\end{proof}

If $T$ is a homogeneous polynomial over $K$ of degree $M$ in $N$ variables and $\bolda\in K^N$ then 
Theorem \ref{ProjectiveLocal} implies that $\nu_v(\bolda,T) = 1$ for almost all places $v$ of $K$.
Hence, we may define the global functions
\begin{equation*}
  \nu(T) = \prod_v\nu_v(T),\quad \mathrm{and}\quad  U(\bolda,T) = \prod_vU_{v}(\bolda,T).
\end{equation*}
which do not depend on $K$.  We now obtain the following projective generalization of \eqref{SamuelsOrig}.

\begin{thm} \label{ProjectiveMain}
  If $\bolda\in\alg^N$ then
  \begin{equation} \label{ProjectiveMainEq}
    1 = H(\bolda)^M\cdot U(\bolda,T)
  \end{equation}
  holds for all homogeneous polynomials $T$ over $\alg$ of degree $M$ in $N$ variables having $T(\alpha)\ne 0$.
\end{thm}
\begin{proof}
  Suppose that $K$ is a number field containing the entries of $\bolda$ and the coefficients of $T$.  Hence, we
  may view $\bolda$ as an element of $\Omega_v^N$ and $T$ as an element of $\mathcal L_v$ for all places
  $v$ of $K$.  Thus, Lemma \ref{ProjectiveLocal} implies that
  \begin{equation} \label{LocalApp}
    |T(\bolda)|_v = H_v(\bolda)^M\cdot U_v(\bolda,T)
  \end{equation}
  at every place $v$ of $K$.  The result follows by taking the product of \eqref{LocalApp} over all places of $K$ 
  and applying the product formula to $T(\bolda)$.
\end{proof}

We may construct applications of Theorem \ref{ProjectiveMain} that are similar to those found in \cite{Samuels}.
Suppose, for example, that $F$ is an homogeneous polynomial of degree $M$ in $N$ variables with coefficients in $\intg$.  
Let $\mathcal X(F)$ denote the subvariety of $\proj^{N-1}(\alg)$ consisting of all points $\bolda$ with $F(\bolda)=0$.
Suppose further that $T$ is homogeneous of degree $M$ in $N$ variables and that $m\in \intg$ are such that
$$T\equiv F\mod m.$$  That is, the coefficients of $T$ are congruent to the coefficients of $F$ modulo $m$.
If $\bolda\in\mathcal X(T)$ then Theorem \ref{ProjectiveMain} implies that
\begin{equation*}
  1= H(\bolda)^M\cdot U(\bolda, T).
\end{equation*}
Now select a number field $K$ containing the entries of $\bolda$.
If $v$ is non-Archimedean then
\begin{equation*}
  U_v(\bolda,T) \leq \nu_v(T-F) \leq |m|_v
\end{equation*}
so that
\begin{equation*}
  U(\bolda,T) \leq \nu_\infty(T)\prod_{v\nmid\infty}|m|_v = m^{-1}\cdot\nu_\infty(T).
\end{equation*}
If $T$ has coefficients $c_1,\ldots,c_R\in\intg$ define
\begin{equation*}
  L_\infty^1(T) = \left(\sum_{r=1}^R \|c_r\|_v\right)^{d_v/d}
\end{equation*}
and note that by the triangle inequality we have that $\nu_\infty(T)\leq L_\infty^1(T)$.  
Hence, we obtain a lower bound on the projective height of $\bolda$
\begin{equation} \label{CongruentLowerBound}
  H(\bolda)^{\deg F} \geq \frac{m}{L_\infty^1(T)}.
\end{equation}
for all $\bolda\in \mathcal X(F)\setminus \mathcal X(T)$.  Hence, if $L_\infty^1(T)$ is small relative to $m$
then we obtain a uniform lower $H(\bolda)^{\deg F}$ over all $\bolda\in \mathcal X(F)\setminus \mathcal X(T)$.
In particular, if $T$ is a monomial having coefficient $\pm 1$ then \eqref{CongruentLowerBound} becomes
\begin{equation*}
  H(\bolda)^{\deg F} \geq m
\end{equation*}
which is non-trivial for all $m\geq 2$.

\section{The height on subspaces using auxiliary linear transformations}\label{Matrices}

We now turn our attention to the height on subspaces and attempt to construct an analog Theorem
\ref{ProjectiveMain}.  Suppose that $X$ is an $N$-dimensional vector space over $\Omega_v$ and fix a basis
$\{\bolde_1,\ldots,\bolde_N\}$ for $X$.  For ease of notation, we identify $X$ with $\Omega_v^N$ by writing
\begin{equation*}
  \boldx = x_1\bolde_1 +\cdots + x_N\bolde_N = (x_1,\ldots, x_N).
\end{equation*}
In this way, we obtain the projective height of $\boldx\in X$ by
\begin{equation*}
  H_v(\boldx) = \max\{|x_1|_v,\ldots,|x_M|_v\}.
\end{equation*}
Of course, this is a norm on $X$, and therefore, it yields the natural dual norm of an element $\phi\in X^\ast$
\begin{equation*}
  \nu_v(\phi) = \sup\{|\phi(\boldx)|_v:\boldx\in X,\ H_v(\boldx)\leq 1\}.
\end{equation*}
Now fix an element $\boldw\in X$ and let 
\begin{equation*}
  S^\ast(\boldw) = \{\phi\in X^\ast:\phi(\boldw) = 0\}
\end{equation*}
so that $S^\ast(\boldw)$ is an $N-1$ dimensional subspace of $X^\ast$.  Finally, for $\psi\in X^\ast$ we set
\begin{equation*}
  U_v(\boldw,\psi) = \inf\{\nu_v(\psi-\phi):\phi\in S^\ast(\boldw)\}.
\end{equation*}
We note that this defines a norm on the one dimensional quotient $X^\ast/S^\ast(\boldw)$.  Of course, this implies that
the ratio $|\psi(\boldw)|_v/U_v(\boldw,\psi)$ depends only on $\boldw$ and $v$.  Analogous to the results of
\cite{Samuels} and the results of section \ref{Homogeneous} we are able to determine this ratio precisely.

\begin{lem} \label{LocalRegular}
  If $\boldw\in X$ then
  \begin{equation*}
    |\psi(\boldw)|_v = H_v(\boldw)\cdot U_v(\boldw,\psi)
  \end{equation*}
  holds for all $\psi\in X^\ast$.
\end{lem}
\begin{proof}
  If $\psi(\boldw)=0$ then both sides of the deisred identity equal $0$.  Hence, we assume without loss
  of generality that $\psi(\boldw)\ne 0$.  Let $\boldw=(w_1,\ldots,w_N)$ and we select an integer $n$
  such that $H_v(\boldw) = |w_n|_v$.  Of course, $w_n\ne 0$ and $H_v(\boldw/w_n)=1$ so that we obtain
  \begin{equation*}
    |\psi(\boldw)|_v = |w_n|_v\cdot|\psi(\boldw/w_n)|_v \leq H_v(\boldw)\cdot\nu_v(\psi)
  \end{equation*}
  for all $\psi\in X^\ast$.  Hence, if $\phi\in S^\ast(\boldw)$ then 
  \begin{equation*}
    |\psi(\boldw)|_v = |(\psi-\phi)(\boldw)|_v \leq H_v(\boldw)\cdot\nu_v(\psi-\phi)
  \end{equation*}
  Taking the infimum of the right hand side over all $\phi\in S^\ast(\boldw)$ we obtain
  \begin{equation}\label{LEQ}
    |\psi(\boldw)|_v \leq H_v(\boldw)\cdot U_v(\boldw, \psi).
  \end{equation}

  We now attempt to prove the opposite inequality.  We define the map $J:X^\ast\to X$ by
  \begin{equation*}
    J(\phi) = (\phi(\bolde_1),\ldots,\phi(\bolde_N))
  \end{equation*}
  and note that $J$ is a vector space isomorphism having the property that $\phi(\boldw) = J(\phi)\cdot\boldw$
  where $\cdot$ represents the inner product.  We now define appropriate bases for $X^\ast$ and $S^\ast(\boldw)$.
  Let $\boldc_n = (0,\ldots,0,w_n^{-1},0,\ldots,0)^T$ and note that $c_n\cdot\boldw = 1$.  For each index $k\ne n$,
  we define $\boldc_k$ in the following way.  If $w_k\ne 0$ then we let $\boldc_k$ be the vector having $w_k^{-1}$ as
  the $k$th entry and $-w_n^{-1}$ as the $n$th entry.  If $w_k=0$ then we let $\boldc_k$ be the vector having
  $1$ as the $k$th entry and zero elsewhere.  Hence, $\{J^{-1}(\boldc_1),\ldots,J^{-1}(\boldc_N)\}$ forms 
  a basis for $X^\ast$ and 
  $$\{J^{-1}(\boldc_1),\ldots,J^{-1}(\boldc_{n-1}),J^{-1}(\boldc_{n+1}),\ldots,J^{-1}(\boldc_N)\}$$ forms a basis
  for $S^\ast(\boldw)$.

  Now write $\psi = \psi_1J^{-1}(\boldc_1) + \cdots +  \psi_NJ^{-1}(\boldc_N)$ and note that $\psi(\boldw) = \psi_n$.  
  Therefore,
  \begin{align*}
    U_v(\boldw,\psi) & = \inf\{\nu_v(\psi-\phi):\phi\in S^\ast(\boldw)\} \\
    & = \inf\{\nu_v(\psi_1J^{-1}(\boldc_1) + \cdots +  \psi_NJ^{-1}(\boldc_N) - \phi):\phi\in S^\ast(\boldw)\} \\
    & = \inf\{\nu_v(\psi_nJ^{-1}(\boldc_n) - \psi_n\phi): \phi\in S^\ast(\boldw)\} \\
    & = |\psi_n|_v\cdot U_v(\boldw,J^{-1}(\boldc_n)) \\
    & = |\psi(\boldw)|_v\cdot U_v(\boldw,J^{-1}(\boldc_n))
  \end{align*}
  Next, we observe that
  \begin{align*}
    U_v(\boldw,J^{-1}(\boldc_n)) & \leq \nu_v(J^{-1}(\boldc_n)) \\
    & = \sup\{|c_n\cdot\boldz|_v:H_v(\boldz)\leq 1\} \\
    & = |w_n|_v^{-1} \\
    & = H_v(\boldw)^{-1}.
  \end{align*}
  We have found that
  \begin{equation*}
    U_v(\boldw,\psi) \leq |\psi(\boldw)|_v\cdot H_v(\boldw)^{-1}
  \end{equation*}
  and the result follows from \eqref{LEQ}.
\end{proof}
 
In order to generalize Lemma \ref{LocalRegular} to include the height on subspaces rather than simply the
projective height, we must now consider the $M$th exterior power $\wedge^M(\Omega_v^N)$.  We define the index set
\begin{equation*}
  \mathcal I_M = \{I\subset\{1,2,\ldots,N\}:|I|=M\}.
\end{equation*}
If $\{\bolde_1,\ldots,\bolde_N\}$ is the standard basis for $\Omega_v^N$, we obtain a natural basis 
\begin{equation} \label{WedgeBasis}
  \left\{\bigwedge_{i\in I}\bolde_i: I\in\mathcal I\right\}
\end{equation}
for $\wedge^M(\Omega_v^N)$ over $\Omega_v$.  The height of an element 
$\boldx\in\wedge^M(\Omega_v^N)$ is computed using the basis \eqref{WedgeBasis}.
For $\phi$ belonging to the dual $(\wedge^M(\Omega_v^N))^\ast$, the norm of $\phi$ is given by
\begin{equation*}
  \nu_v(\phi) = \sup\{|\phi(\boldx)|_v:\boldx\in \wedge^M(\Omega_v^N),\ H_v(\boldx)\leq 1\}.
\end{equation*}
If $\boldw\in\wedge^M(\Omega_v^N)$ then
\begin{equation*}
  U_v(\boldw,\psi) =  \inf\{\nu_v(\psi-\phi):\phi\in(\wedge^M(\Omega_v^N))^\ast,\ \phi(\boldw) = 0\}.
\end{equation*}
We also obtain the following lemma showing that a surjective linear transformation $\Psi:\Omega_v^N\to\Omega_v^M$
may be viewed as a map on $\wedge^M(\Omega_v^N)$.

\begin{lem} \label{UniversalProperty}
  Suppose that $\Psi:\Omega_v^N\to\Omega_v^M$ is a surjective linear transformation.  Then there exists
  a unique linear transformation $\wedge^M(\Psi):\wedge^M(\Omega_v^N) \to \Omega_v$ such that
  \begin{equation*}
    \wedge^M(\Psi)(\boldw_1\wedge\cdots\wedge\boldw_M) = \det\left(\begin{array}{c}
        \Psi(\boldx_1) \\ \vdots \\ \Psi(\boldx_M) \end{array}\right)
  \end{equation*}
  for all $\boldw_1,\ldots,\boldw_M\in \Omega_v^N$.
\end{lem}
\begin{proof}
  Let $\mathcal M_{M\times M}(\Omega_v)$ denote the vector space of $M\times M$ matrices with entries in $\Omega_v$.
  We note that $\Psi$ induces a unique $M$-multilinear map $\Psi':(\Omega_v^N)^M \to M_{M\times M}(\Omega_v)$ 
  given by
  \begin{equation*}
    \Psi'(\boldw_1,\ldots,\boldw_M) = \left(\begin{array}{c}
        \Psi(\boldx_1) \\ \vdots \\ \Psi(\boldx_M) \end{array}\right).
  \end{equation*}
  Furthermore, it is well-known (see, for example, \cite{DummitFoote}, p. 437) that the determinant map 
  $\det:\mathcal M_{M\times M}(\Omega_v)\to \Omega_v$ defines an $M$-multilinear map on the rows of the 
  elements in $\mathcal M_{M\times M}(\Omega_v)$.  Hence, we conclude that the composition $\det\circ\Psi'$ is an
  $M$-multilinear map from $(\Omega_v^N)^M$ to $\Omega_v$.  Moreover, if there exist $i\ne j$ with
  $\boldw_i = \boldw_j$ then 
  \begin{equation*}
    \det\circ\Psi'(\boldw_1,\ldots,\boldw_M) = 0
  \end{equation*}
  It follows that $\det\circ\Psi'$ is, in fact, an alternating $M$-multilinear map.
  
  By the universal property for alternating $M$-tensors, there exists a unique linear transformation
  $T:\wedge^M(\Omega_v^N) \to \Omega_v$ such that 
  \begin{equation*}
    T\circ\iota = \det\circ\Psi'
  \end{equation*}
  where
  $\iota:(\Omega_v^N)^M \to \wedge^M(\Omega_v^N)$ is given by
  \begin{equation*}
    \iota(\boldw_1,\ldots,\boldw_M) = \boldw_1\wedge\cdots\wedge\boldw_M.
  \end{equation*}
  Therefore, we conlude that
  \begin{align*}
    T(\boldw_1\wedge\cdots\wedge\boldw_M) & = T(\iota(\boldw_1,\ldots,\boldw_M)) \\
    & = \det(\Psi'(\boldw_1,\ldots,\boldw_M)) \\
    & = \det\left(\begin{array}{c}
        \Psi(\boldx_1) \\ \vdots \\ \Psi(\boldx_M) \end{array}\right).
  \end{align*}
  By taking $\wedge^M(\Psi)= T$ we complete the proof.
\end{proof}

We now assume that $W$ is an $M$-dimensional subspace of $\alg^N$ and $\Psi:\alg^N\to\alg^M$ is a
surjective linear transformation.  Select a basis $\{\boldw_1,\ldots,\boldw_M\}$ for $W$ and assume that $K$ is a
number field containing the entries of each basis element $\boldw_m$ as well as the entries of $\Psi$.
We note that the height of $W$ is given by
\begin{equation*}
  H(W) = \prod_vH_v(\boldw_1\wedge\cdots\wedge\boldw_M)
\end{equation*}
where the product is taken over all places $v$ of $K$.  As we noted in our introduction, the product formula
implies that this definition does not depend on the choice of basis for $W$.  By Lemma
\ref{UniversalProperty} we may define 
\begin{equation} \label{SubspaceQuotientNorm}
  U(W,\Psi) = \prod_vU_v(\boldw_1\wedge\cdots\wedge\boldw_M,\wedge^M(\Psi)).
\end{equation}
Lemma \ref{LocalRegular} shows that this product is indeed finite and, by the way we have normalized our absolute values,
it does not depend on $K$.  As in the height on subspaces, the product formula implies that 
\eqref{SubspaceQuotientNorm} is independent of the basis for $W$ as well.  We may now state and prove our 
main result.

\begin{thm}\label{SubspaceHeightTheorem}
  If $W$ is an $M$-dimensional subspace of $\alg^N$ then
  \begin{equation*}
    1= H(W)\cdot U(W,\Psi)
  \end{equation*}
  holds for all surjective linear transformations $\Psi:\alg^N\to\alg^M$ with 
  $W\cap\ker\Psi = \{{\bf 0}\}$.
\end{thm}
\begin{proof}
  Let $\{\boldw_1,\ldots,\boldw_M\}$ be a basis for $W$ and let $K$ be a number field containing
  the entries of each basis element $\boldw_m$ and the entries of $\Psi$.  Hence, $\boldw_m\in\Omega_v^N$
  and $\Psi:\Omega_v^N\to\Omega_v^M$ for all places $v$ of $K$.  Therefore, Lemma \ref{LocalRegular} implies that
  \begin{align} \label{WedgeLocal}
    |\wedge^M(\Psi)&(\boldw_1\wedge\cdots\wedge\boldw_M)|_v \nonumber \\
    &= H_v(\boldw_1\wedge\cdots\wedge\boldw_M)
    \cdot U_v(\boldw_1\wedge\cdots\wedge\boldw_M,\wedge^M(\Psi)).
  \end{align}
  By Lemma \ref{UniversalProperty} we have that
  \begin{equation*}
    \wedge^M(\Psi)(\boldw_1\wedge\cdots\wedge\boldw_M) = \det\left(\begin{array}{c}
        \Psi(\boldw_1) \\ \vdots \\ \Psi(\boldw_M) \end{array}\right).
  \end{equation*}
  Since $W\cap\ker\Psi = \{{\bf 0}\}$ we know that the rows in the above matrix are linearly independent
  so that its determinant is non-zero.  Hence, the left hand side of \eqref{WedgeLocal} is non-zero and we
  may apply the product formula.  The desired identity follows immediately.
\end{proof}

It is natural to consider the special case of Theorem \ref{SubspaceHeightTheorem} in which $W$ is a
one dimensional subspace spanned by an element $\boldw\in\alg^N$.  For $\psi\in(\alg^N)^\ast$ we define
\begin{equation*}
  U(\boldw,\psi) = \prod_v U_v(\boldw,\psi)
\end{equation*}
and obtain the following corollary.

\begin{cor} \label{SingleVectorHeight}
  If $\boldw\in\alg^N$ then
  \begin{equation*}
    1= H(\boldw)\cdot U(\boldw,\psi)
  \end{equation*}
  for all $\psi\in (\alg^N)^\ast$ with $\psi(\boldw)\ne 0$.
\end{cor}
\begin{proof}
  If $W$ is the one dimensional subspace spanned by $\boldw$ then it is easy to see that $H(W) = H(\boldw)$.
  Furthermore, $\psi:\alg^N\to\alg$ is a linear transformation and $U(W,\psi) = U(\boldw,\psi)$.
  Theorem \ref{SubspaceHeightTheorem} yields that $1=H(W)\cdot U(W,\psi)$ and the result follows immediately.
\end{proof}

\end{document}